\documentclass[a4paper,12pt]{amsart}
\usepackage{amssymb,amsfonts,amsthm}
\usepackage{amstext}
\usepackage{comment}
\usepackage{amsmath}
\usepackage{hyperref}  

\usepackage{color}
\usepackage[latin1]{inputenc}
\usepackage[active]{srcltx}
\usepackage{mathrsfs} 
\usepackage{comment} 
\usepackage{graphicx}
\usepackage{amssymb,amsfonts,amstext, amsmath}
\usepackage[latin1]{inputenc}
\usepackage{tensor}
\usepackage{color}
\usepackage{mathptmx}

\newtheorem{thm}{Theorem}[section]
\newtheorem{cor}[thm]{Corollary}
\newtheorem{lem}[thm]{Lemma}

\newtheorem{defn}[thm]{Definition}

\DeclareMathAlphabet\EuScript{U}{eus}{m}{n}
\SetMathAlphabet\EuScript{bold}{U}{eus}{b}{n}

\newtheorem*{rep@theorem}{\rep@title}
\newcommand{\newreptheorem}[2]{%
	\newenvironment{rep#1}[1]{%
		\def\rep@title{#2 \ref{##1}}%
		\begin{rep@theorem}}%
		{\end{rep@theorem}}}
\makeatother

\theoremstyle{definition}

\DeclareMathOperator{\arccosh}{arcCosh}

\usepackage{chngcntr} 

\usepackage{mathrsfs} 

\usepackage{comment} 

\usepackage{hyperref}  
\begin{document}
	
	\title{Generalization of the energy distance by Bernstein functions}
	
	\author{J. C. Guella}
	\email{jean.guella@riken.jp}
	\address{RIKEN Center for Advanced Intelligence Project, Tokyo, Japan}
	
	\begin{abstract}We reprove the well known fact that the energy distance defines a metric on the space of Borel probability measures on a  Hilbert space with finite first moment by a new approach, by analysing the behaviour of the Gaussian kernel on Hilbert spaces and a Maximum Mean Discrepancy analysis. From this new point of view we are able to generalize the energy distance metric to a family of kernels  related to Bernstein functions and conditionally negative definite kernels. We also explain what occurs on the energy distance on the kernel $\|x-y\|^{\alpha}$ for every $\alpha >2$, where we also generalize the idea to a family of  kernels  related to derivatives of completely monotone functions and conditionally negative definite kernels.    
	\end{abstract}

	\maketitle

	\tableofcontents
	
\section{Introduction}


A popular method to compare two probabilities is done by  embedding the space (or a subset) of probabilities into a Hilbert space and use the metric provided by the embedding. Currently, there are two main approaches  for this task:
\begin{enumerate}
    \item[$(I)$]  The maximum mean discrepancy on a bounded, continuous, positive definite kernel $K: X\times X \to \mathbb{R}$ that is characteristic \cite{gretton2006kernel}, \cite{fukumizu2004dimensionality}. The distance between two Radon regular probabilities $P$ and $Q$ is defined by  
    $$
    MMD(P,Q):= \sqrt{\int_{X}\int_{X}K(x,y)d[P-Q](x)d[P-Q](y) }.
    $$
   \item[$(II)$]  The use of a continuous conditionally negative definite kernel $\gamma: X \times X \to \mathbb{R}$ with $\gamma(x,x)=0$ for every $x \in X$, \cite{sejdinovic2013equivalence}. The kernel $\gamma$ must additionally satisfy the equality
 \begin{equation}\label{condcndkernel}
   \int_{X}\int_{X}-\gamma(x,y)d[P-Q](x)d[P-Q](y)= 0 
\end{equation}
   for two Radon regular probabilities $P$ and $Q$ that integrates the function $x \to  \gamma(x,z)$ for every $z \in X$ only when $P=Q$. It can be proved that the above double integral is always a nonnegative number and when this property occurs 
    $$
   D_{\gamma}(P,Q):= \sqrt{\int_{X}\int_{X}-\gamma(x,y)d[P-Q](x)d[P-Q](y) },
    $$ 
    is a metric on the mentioned subspace of probabilities on $X$. 
\end{enumerate}

On this paper, we focus on the second method.

The most popular example of this method is the energy distance, initially defined as $X=\mathbb{R}^{m}$, $\gamma(x,y)= \|x-y\|^{\theta}$, where $0< \theta < 2$  and the set of probabilities are those that integrates  $\|x\|^{\theta}$, \cite{szekely2004testing}, \cite{szekely2013energy}. When $\theta=2$, the kernel is conditionally negative definite but do not satisfy the additional property of Equation \ref{condcndkernel}. 

A more geometrical approach is when $\gamma$ is a metric on $X$ that satisfy  Equation \ref{condcndkernel} (the topology is the one from the metric), hence $(X, \gamma)$ is a metric space of strong negative type. Examples of such spaces include:

$\bullet$ Hilbert spaces: Proved on \cite{lyons2013} as a generalization of the energy distance. 

$\bullet$ Hyperbolic spaces (finite dimensional): Proved on \cite{lyons2014}

In some cases, the conditionally negative definite kernel  $\gamma$ may define a metric on the set $X$, but $\gamma$ is not of strong type. A metric space where we only know that the distance is a conditionally negative definite kernel is called a metric space of negative type.  An example of such space is the real sphere, proved on  \cite{gangolli}, where it is also proved that the real, complex and quaternionic projective spaces and the Cayley projective plane are not metric spaces of negative type.

In \cite{lyons2013}, it is also proved that if $(X,\gamma)$ is a metric space of  negative type  then $\gamma^{\theta}$, $0< \theta <1$ is a conditionally negative definite kernel that satisfies  Equation \ref{condcndkernel}, with the topology of the metric $\gamma$.  Interestingly, the kernel $\gamma^{\theta}$ is a metric on $X$, with the same topology as $\gamma$, so we can rephrase the result of Lyon as $( X, \gamma^{\theta})$ being a metric space of strong negative type. We provide more details and generalizations of this property on Corollary \ref{ber+condcor}.

The major aim of this paper is to provide a large amount of examples of conditionally negative definite kernels that satisfy Equation \ref{condcndkernel}, by using Bernstein  functions on Theorem \ref{ber+cond}. Our method encompasses  all of the above mentioned kernels that satisfy $(II)$. We also provide a new proof that hyperbolic spaces (any dimension)  are metric paces of strong negative type  on Theorem on \ref{hyperinnprod}.

In \cite{mattner1997strict}, Mattner analysed the behaviour of the kernel $\|x-y\|^{\alpha}$, for $\alpha >2$, defined on  $\mathbb{R}^{m}$. What occurs is that we can still provide a metric structure on the space of probabilities with certain integrability assumptions,  but we can only compare them if they have the same vector mean ( $2< \alpha < 4$), the same same vector mean and the same covariance matrix ($4< \alpha < 6$), and so on. It also provided the same analysis for others  radial kernels, that we generalize on Theorem \ref{principal} to a broader setting.

Section \ref{Conditionally positive definite kernels} is focused on the integrability conditions of a conditionally negative definite kernel (and its generalizations). Section \ref{Inner products defined by CND kernels and derivatives of completely monotone functions} contains the most important results of this paper, mentioned before. On Section \ref{Equimeasurability for derivatives of completely monotone functions} we analyse the space of functions
$$
y \in \mathbb{H}\to \int_{\mathcal{H}} \psi(\|x-y\|^{2}) d\mu(x) \in \mathbb{R},
$$
 where $\psi$ is a continuous function that is the difference of two derivatives (same order) of a completely monotone function. More precisely, we analyse when they are uniquely defined by the measure $\mu$. Section \ref{Definitions} is entirely focused on definitions that we use. The proofs are presented on Section \ref{Proofs}.

\section{Definitions}\label{Definitions}

We recall that a  nonnegative measure $\lambda$ on a  Hausdorff space $X$ is    Radon regular (which we simply refer as Radon) when it is a Borel measure such that is finite on every compact set of $X$ and
\begin{enumerate}
    \item[(i)](Inner regular)$\lambda(E)= \sup\{\lambda(K), \ \ K \text{ is compact }, K\subset E\} $ for every Borel set $E$.
    \item[(ii)](Outer regular) $\lambda(E)= \inf\{\lambda(U), \ \ U \text{ is open }, E\subset U\} $ for every Borel set $E$.
\end{enumerate}
 
 We then said that a complex valued measure $\lambda$ of bounded variation is Radon if its variation is a Radon measure. The vector space of such measures is denoted by $\mathfrak{M}(X)$. Recall that every Borel measure of finite variation (in particular, probability measures) on a separable complete metric space is necessarily Radon. 
 
 An semi-inner product on a real (complex) vector space $V$  is a bilinear  real (sesquilinear complex) valued function $( \cdot, \cdot)_{V} $ defined on  $V\times V$ such that $(u,u)_{V} \geq 0$ for every $u \in V$. When this inequality is an equality only for $u=0$, we say that  $( \cdot, \cdot)_{V}$ is an inner-product. Similarly, a pseudometric on a set $X$ is a symmetric function $d: X\times X \to [0, \infty)$, such that $d(x,x)=0$ that satisfies the triangle inequality. If $d(x,y)=0$ only when $x=y$, $d$ is a metric on $X$. 
 
 A kernel $K: X \times X \to \mathbb{C}$ is called positive definite if for every finite quantity of distinct points $x_{1}, \ldots, x_{n} \in X$ and scalars $c_{1}, \ldots, c_{n} \in \mathbb{C}$, we have that
$$
\int_{X}\int_{X}K(x,y)d\lambda(x)d\overline{\lambda}(y)=\sum_{i, j =1}^{n}c_{i}\overline{c_{j}} K(x_{i}, x_{j}) \geq 0,
$$
where $\lambda= \sum_{i=1}^{n}c_{i}\delta_{x_{i}}$.  The set of measures on $X$ used before are denoted by the symbol $\mathcal{M}_{\delta}(X)$.

 The  reproducing kernel Hilbert space (RKHS) of a positive definite kernel $K: X \times X \to \mathbb{C}$ is the Hilbert space $\mathcal{H}_{K} \subset \mathcal{F}(X, \mathbb{C})$, and it satisfies \cite{Steinwart}
 \begin{enumerate}
    \item[$(i)$]  $x \in X \to K_{y}(x):= K(x,y) \in \mathcal{H}_{K}$;
    \item[$(ii)$] $\langle K_{y}, K_{x}\rangle = K(x,y) $
    \item[$(iii)$] $\overline{ span\{ K_{y}, \quad y \in X\} }= \mathcal{H}_{K}$.
    \end{enumerate}
When  $X$ is a Hausdorff  space and $K$ is continuous it holds that $\mathcal{H}_{K} \subset C(X)$.

The following widely known result describes how it is possible to define a semi-inner product structure on a subspace of $\mathfrak{M}(X)$ using a continuous positive definite kernel.
 
 \begin{lem}\label{initialextmmddominio}
If $K: X \times X \to \mathbb{C}$ is a continuous positive definite kernel and  $\mu \in \mathfrak{M}(X)$  with $\sqrt{K(x,x)} \in L^{1}(|\mu|)$ ($\mu \in \mathfrak{M}_{\sqrt{K}}(X)$), then 
$$
z \in X \to K_{\mu}(z):=\int_{X} K(x,z)d\mu(x) \in \mathbb{C}
$$	
is an element of $\mathcal{H}_{K}$, and if $\eta$ is another measure  with the same conditions as $\mu$, we have that
$$
\langle K_{\eta}, K_{\mu}\rangle_{\mathcal{H}_{K}}= \int_{X} \int_{X}k(x,y)d\eta(x)d\overline{\mu}(y).
$$
In particular, $(\eta, \mu) \in \mathfrak{M}_{\sqrt{K}}(X) \times \mathfrak{M}_{\sqrt{K}}(X) \to \langle K_{\eta}, K_{\mu}\rangle_{\mathcal{H}_{K}} $ is a semi-inner product.
\end{lem}

 We present a generalization of this result to a larger class of measures in Lemma \ref{extmmddominio}. Usually, the kernel $K$ is bounded, so $\mathfrak{M}_{\sqrt{K}}(X) = \mathfrak{M}(X)$. On this case, if the semi-inner product is in fact an inner product we say that $K$ is integrally strictly positive definite (ISPD), and when is an inner product on the vector space of measures in $\mathfrak{M}(X)$ that $\mu(X)=0$, we say that $K$ is characteristic. If the kernel $K$ is real valued, it is sufficient to analyse the ISPD and characteristic property on real valued measures.
 
 When the kernel is characteristic we define the maximum mean discrepancy (MMD) as the metric on  the space of probability measures in $\mathfrak{M}(X)$ by 
 \begin{equation}\label{MMD}
 MMD(P,Q)_{K} := \sqrt{ \langle  K_{P} - K_{Q}, K_{P} - K_{Q}\rangle_{\mathcal{H}_{K}} }= \sqrt{\int_{X} \int_{X}K(x,y)d[P-Q](x)d[P-Q](y)}
 \end{equation}
 

As mentioned at the introduction, the focused of this paper is to analyse metrics on the space of probabilities using conditionally negative definite kernels.  We present a more general definition which will be useful to the analysis of the energy distance through the kernel $\|x-y\|^{\alpha}$, $\alpha >2$, defined  on a Hilbert space.
 
   \begin{defn}\label{P-PD} Let  $\gamma: X \times X \to \mathbb{C}$ be an Hermitian kernel and $P$ a finite dimensional space of functions from $X$ to $\mathbb{C}$. We say that $\gamma$ is $P$-conditionally positive definite ($P$-CPD) if for every finite quantity of points $x_{1}, \ldots , x_{n} \in X$ and scalars $c_{1}, \ldots, c_{n} \in \mathbb{C}$, under the restriction that $\sum_{i=1}^{n}c_{i}p(x_{i})=0$ for every $p \in P$, we have that
	$$
	\sum_{i,j=1}^{n}c_{i}\overline{c_{j}}\gamma(x_{i}, x_{j})\geq 0.
	$$
	\end{defn}
	
	This definition generalize the concepts of positive definite kernels ($P$ is the zero space) and CPD kernels ($P$ as the set of constant functions). The most important example is when $X$ is a finite dimensional Euclidean space and $P$ is the set of multivariable polynomials on $X$ with degree less than or equal to a constant $k \in  \mathbb{N}$, \cite{wendland} \cite{guo}, \cite{gelfand}. Sometimes it might be more convenient to work with the opposite sign on Definition \ref{P-PD}, on this case we say that the kernel is $P$-conditionally negative definite ($P$-CND).

In \cite{guo}, \cite{michdistance}, it is proved that a characterization for the continuous functions $\psi: [0, \infty) \to \mathbb{R}$,  such that the kernel 
$$ 
(x,y) \in \mathbb{R}^{m} \times \mathbb{R}^{m} \to \psi(\|x-y\|^{2}) \in \mathbb{R}
$$
is CPD for $P$ as the family of multivariable polynomials of degree less than  a fixed $\ell \in \mathbb{Z}_{+}$ (we denote this family by $\pi_{\ell-1}(\mathbb{R}^{m})$, where $\pi_{-1}(\mathbb{R}^{m})=\{0\}$ and $\pi_{0}(\mathbb{R}^{m})= \{\text{constant functions}\}$) for every $m \in \mathbb{N}$. A function $\psi$ satisfy this property  if and only if $\psi \in C^{\infty}(0, \infty)$ and $(-1)^{\ell}\psi^{(\ell)}$ is a completely monotone function on $(0, \infty)$. A function with this property can be uniquely written as
\begin{equation}\label{compleelltimes}
 \psi(t)= \int_{(0,\infty)} \frac{e^{-tr} - e_{\ell}(r)\omega_{\ell,\infty}(rt)}{r^{\ell}} d\lambda(r) + \sum_{k=0}^{\ell}a_{k}t^{k}
\end{equation}  
where $\lambda $ is a nonnegative Radon measure on $(0,\infty)$ (not necessarily with finite variation) with
$$
\omega_{\ell,\infty} (s):= \sum_{l=0}^{\ell-1}(-1)^{l}\frac{s^{l}}{l!}, \quad e_{\ell}(s):= e^{-s}\sum_{l=0}^{\ell-1}\frac{s^{l}}{l!}, \quad \int_{(0,\infty)}\min \{1, r^{-\ell}\}d\lambda(r)<\infty
$$
and $a_{k} \in \mathbb{R}$, $(-1)^{\ell}a_{\ell} \geq 0$ and $\omega_{0, \infty}$ is the zero function. For instance, the functions
\begin{enumerate}
    \item[$i)$] $(-1)^{\ell}t^{a} + p(t)$;
    \item[$ii)$] $(-1)^{\ell +1}t^{\ell}\log (t) +p(t)$;
    \item[$iii)$] $(-1)^{\ell}( c+t )^{a} +p(t)$;
    \item[$iv)$] $e^{-rt} + p(t)$, 
\end{enumerate}
 are elements of $CM_{\ell}$, for $\ell-1 < a \leq \ell$ , $c>0$ and $p \in \pi_{\ell-1}$. Those functions are not only in $CM_{\ell}$, but they are $\ell-1$ continuously differentiable on $[0, \infty)$ and we have a similar and simpler characterization compared to Equation \ref{compleelltimes} for them.

In general, a function $\psi \in CM_{\ell}$ is such that $\psi \in C^{\ell-1}([0, \infty))$ if and only if
\begin{equation}\label{compleelltimes2}
 \psi(t)= \int_{(0,\infty)} \frac{e^{-tr} - \omega_{\ell,\infty}(rt)}{r^{\ell}} d\eta(r) + \sum_{k=0}^{\ell}b_{k}t^{k}
\end{equation}  
where $\eta $ is a nonnegative Radon measure on $(0,\infty)$ (not necessarily with finite variation) with
$$
\omega_{\ell,\infty} (s):= \sum_{l=0}^{\ell-1}(-1)^{l}\frac{s^{l}}{l!}, \quad  \int_{(0,\infty)}\min \{1, r^{-\ell}\}d\eta(r)<\infty,
$$
 $b_{k} = \psi^{(k)}(0)/k!$ for $k < \ell$  and $(-1)^{\ell}b_{\ell} \geq 0$. 

Note that if a function $\psi \in CM_{\ell}$ then $\psi( \cdot +c) \in CM_{\ell} \cap C^{\ell-1}([0, \infty))$. On this case, the measure $\eta_{c}$ relative to  the decomposition given on Equation \ref{compleelltimes2} has finite variation  and satisfy $d\eta_{c+s}(r)=e^{-sr}d\eta_{c}(r)$ for every $c,s>0$. This property and the decomposition given on Equation \ref{compleelltimes2} are implicitly proved on Theorem $2.1$ on \cite{michdistance} and can also be found on Theorem $8.19$ of \cite{wendland}. We remark that a polynomial $p \in CM_{\ell}$ if and only if $p \in \pi_{\ell}(\mathbb{R})$ and the constant $(-1)^{\ell}p^{(\ell)}\geq 0$.

By Lemma $2.4$ in \cite{guo}, a function $\psi \in CM_{\ell}$ satisfies $|\psi(t)|\lesssim 1+t^{\ell}$ (this notation means that $|\psi(t)|/1+t^{\ell}$ is a bounded function). 

	\section{Conditionally positive definite kernels}\label{Conditionally positive definite kernels}

	The following known result states a connection between positive definite kernels and  $P$-CPD kernels \cite{wendland}. A Lagrange basis for $P$ is a basis $\{p_{1}, \ldots, p_{m} \}$ of $P$ and points $\xi_{1}, \ldots , \xi_{m} \in X$, such that $p_{i}(\xi_{j})= \delta_{i,j}$. A set of points $\xi_{1}, \ldots , \xi_{m} \in X$ is unisolvent with respect to a  $m$-dimensional space $P$ if the only function $p \in P$ such that $p(\xi_{i})=0$ for every $i$ is the zero function.

		\begin{thm}\label{pontriequi}Let   $\xi_{1}, \ldots , \xi_{m} \in X$ and $p_{1}, \ldots , p_{m}$ be a Lagrange basis  for a  finite dimensional space $P$ of functions from $X$ to $\mathbb{C}$. An Hermitian kernel $\gamma: X \times X \to \mathbb{C}$ is $P$-CPD if and only if the Hermitian kernel
		$$
		K_{\gamma}(x,y):= \gamma(x,y) - \sum_{k=1}^{m}p_{k}(x)\gamma(\xi_{k}, y) - \sum_{l=1}^{m}\overline{p_{l}(y)}\gamma(x, \xi_{l}) + \sum_{k,l=1}^{m}p_{k}(x)\overline{p_{l}(y)}\gamma(\xi_{k}, \xi_{l})
		$$
		is positive definite.
	    \end{thm}
	This result can be easily seen by the fact that if $x_{1}, \ldots , x_{n} \in X$ and $c_{1}, \ldots, c_{n} \in \mathbb{C}$ are such that $\sum_{i=1}^{n}c_{i}p(x_{i})=0$ for every $p \in P$, then
	$$
	\sum_{i,j=1}^{n}c_{i}\overline{c_{j}}K_{\gamma}(x_{i}, x_{j}) = \sum_{i,j=1}^{n}c_{i}\overline{c_{j}} \gamma(x_{i}, x_{j}),
	$$
	and conversely, if   $z_{1}, \ldots , z_{m+n} \in X$ (with $z_{n+k}= \xi_{k}$)  and $d_{1}, \ldots, d_{m+n} \in \mathbb{C}$, then
		$$
	\sum_{i,j=1}^{m+n}d_{i}\overline{d_{j}}K_{\gamma}(z_{i}, z_{j}) = \sum_{ i, j =1}^{m+n}e_{ i}\overline{e_{j}}\gamma(z_{ i}, z_{j}), 
	$$
	where  $e_{ i}= d_{ i}$, for $ i \leq n$ and   $e_{ i}= -\sum_{i=1}^{n}d_{i}p_{ i-n}(z_{i})$, for $ i > n$.

	Similar to continuous positive definite kernels, continuous $P$-CPD kernels can be analyzed by its behaviour  on a certain type of space of measures.
	
	\begin{defn} Let $X$ be a Hausdorff space and $P \subset C(X)$ a finite dimensional vector space. We define the set
	$$
	\mathfrak{M}_{P}(X):=\{ \mu \in \mathfrak{M}(X), \quad \int_{X}|p(x)|d|\mu|(x)< \infty \text{ and } \int_{X}p(x)d\mu(x)=0 \text{ for every } p \in P \}.
	$$
	\end{defn}
	
	\begin{thm}\label{gensupport}  A continuous Hermitian  kernel  $\gamma: X \times X \to \mathbb{C}$  is   $P$-CPD if and only if for every $\mu \in 	\mathfrak{M}_{P}(X)$  for which $\gamma(x,y) \in L^{1}(|\mu|\times |\mu|)$ and $\gamma(x, \xi_{i}) \in L^{1}(|\mu|)$, where $(\xi_{i})_{1\leq i \leq m}$ is unisolvent, we have that
	$$
	\int_{X}\int_{X}\gamma(x,y)d\mu(x)d\overline{\mu}(y) \geq 0.
	$$
	\end{thm}	
If we restrict the measures on Theorem	\ref{gensupport} to those that $\gamma(x, \xi) \in L^{1}(|\mu|)$ for every $\xi \in X$, then the kernel $\gamma$ defines a semi-inner product on this vector space.

When $P$ is the space generated by a single function $p$, we can simplify  the assumptions of  Theorem \ref{gensupport}.

\begin{lem}\label{unidimcase} Let   $\gamma: X \times X \to \mathbb{C}$ be a continuous Hermitian  kernel  and $[p]=P\subset C(X)$ be a one dimensional vector space. Then, $\gamma$   is $P$-CPD if and only if for every $\mu \in 	\mathfrak{M}_{P}(X)$  for which $\gamma(x,y) \in L^{1}(|\mu|\times |\mu|)$ 
	$$
	\int_{X}\int_{X}\gamma(x,y)d\mu(x)d\overline{\mu}(y) \geq 0.
	$$
Additionally, if $p$ and $\gamma$ are real valued functions such that $p(x)\neq 0$  and the function $\gamma(x,x)/p^{2}(x)$ is bounded, the following assertions are equivalent:
\begin{enumerate}
    \item[$(i)$] $\gamma \in L^{1}(|\mu|\times |\mu|)$;
    \item[$(ii)$] The function $x \in X \to \gamma(x,z) \in L^{1}(|\mu|)$ for some $z \in X$;
    \item[$(iii)$] The function $x \in X \to \gamma(x,z) \in L^{1}(|\mu|)$ for every $z \in X$.
    \end{enumerate}
\end{lem}

As a direct consequence of the previous Lemma we obtain that if the function $\gamma(x,x)/p^{2}(x)$ is bounded, the set of measures on 	$\mathfrak{M}_{P}(X)$ that integrates $\gamma(x,y)$ is a vector space and the double integral defines a semi inner product on it. We focus on the CPD case and when $\gamma$ is real valued due to its relevance.

\begin{cor}\label{intequicnd} Let $\gamma: X \times X \to \mathbb{R}$ be a continuous CPD kernel  such that the function $\gamma(x,x)$ is bounded. The semi inner product 
$$
 (\mu ,\nu) \in \mathfrak{M}_{1}(X, \gamma)\times  \mathfrak{M}_{1}(X, \gamma)  \to I(\mu, \nu)_{\gamma}:= \int_{X}\int_{X}\gamma(x,y)d\mu(x)d\nu(y) \in \mathbb{R}
 $$
is well defined on the vector space 
$$ 
\mathfrak{M}_{1}(X, \gamma):=\{\eta \in  \mathfrak{M}(X), \quad \eta(X)=0 , \gamma \in L^{1}(|\eta|\times |\eta|) \}
$$
\end{cor}


On the next lemma we improve the condition $\sqrt{K(x,x)} \in L^{1}(|\mu|)$ and the set of measures analysed on Lemma \ref{initialextmmddominio}, at the cost of describing the function $K_{\mu}$ at the exception of a $|\mu|$ measure zero set.

\begin{lem}\label{extmmddominio} Let $K: X \times X \to \mathbb{C}$ be a continuous positive definite kernel. Let $\mu \in \mathfrak{M}(X)$  such that   $K(x,y) \in L^{1}(|\mu|  \times |\mu|)$, then the set of points 
$$
X_{\mu}:=\{ z \in X, \quad K( \cdot, z ) \in L^{1}(|\mu|)\}
$$
is such that $|\mu|(X - X_{\mu})=0$, and the function
$$
z \in X_{\mu} \to \int_{X} K(x,z)d\mu(x) \in \mathbb{C}
$$
is the restriction of an element $ K_{\mu} \in \mathcal{H}_{K}$. If $\eta$  is a measure   with the same conditions as the measure $\mu$ and $K \in L^{1}(\mu \times \eta )$, we have that
$$
\langle K_{\eta}, K_{\mu}\rangle_{\mathcal{H}_{K}}= \int_{X} \int_{X}k(x,y)d\eta(x)d\overline{\mu}(y).
$$ 
\end{lem}

\section{Inner products defined by CND kernels and derivatives of completely monotone functions}\label{Inner products defined by CND kernels and derivatives of completely monotone functions}

Since all kernels that we deal on this Section are real valued, we simplify the writing by only focusing on real valued measures (which we still use the notation $\mathfrak{M}(X)$). As mentioned on Section $2$, this is not a restriction.  

In \cite{lyons2013}, it is proved that on a separable real Hilbert space $\mathcal{H}$, the bilinear function $I_{1/2}$ defined as 
$$
(\mu, \nu) \in \mathfrak{M}_{1}(\mathcal{H}) \times \mathfrak{M}_{1}(\mathcal{H}) \to I(\mu, \nu)_{1/2}:=  \int_{\mathcal{H}}\int_{\mathcal{H}} -\| x-y\|_{\mathcal{H}}d\mu(x)d\nu(y) 
$$
defines a inner product on the vector space
$$
\mathfrak{M}_{1}(\mathcal{H}):= \{ \eta \in \mathfrak{M}(\mathcal{H}), \quad \eta(\mathcal{H})=0 ,  \|x\| \in L^{1}(|\eta|)\}.
$$

The function $t \in [0, \infty) \to  \psi(t):=\sqrt{t} \in \mathbb{R}$ is an example of a Bernstein function, \cite{bers}. It is   continuous,  $\psi \in C^{\infty}((0, \infty))$ and $\psi^{\prime}$ is a completely monotone function on $(0, \infty)$ (we do not need to assume on our context that Bernstein functions are nonnegative). In other words, a function $\psi$ is a Bernstein function if and only if $-\psi \in CM_{1}$, and then it can be written, by   Equation \ref{compleelltimes2} for $\ell=1$, as
$$
-\sqrt{t} = \frac{1}{2\sqrt{\pi}}\int_{(0, \infty)}(e^{-rt}-1)\frac{1}{r^{3/2}}dr.
$$
So,
$$
(x,y) \in \mathcal{H} \times \mathcal{H} \to -\|x-y\|_{\mathcal{H}}= \frac{1}{2\sqrt{\pi}}\int_{(0, \infty)}(e^{-r\|x-y\|^{2}}-1)\frac{1}{r^{3/2}}dr,
$$
and this kernel is CPD.  The Gaussian kernels $e^{-r\|x-y\|^{2}}$, $r>0$,  are ISPD for every Hilbert space \cite{gaussinfi}, being so, by Fubini-Tonelli Theorem we have that if $\mu \in \mathfrak{M}(\mathcal{H})$ with $\mu(\mathcal{H})=0$ and $\|x\| \in L^{1}(|\mu|)$, then 
$$
  \int_{\mathcal{H}}\int_{\mathcal{H}} (-1)\| x-y\|_{\mathcal{H}}d\mu(x)d\mu (y)= \frac{1}{2\sqrt{\pi}}\int_{(0, \infty)} \left (  \int_{\mathcal{H}}\int_{\mathcal{H}} e^{-r\| x-y\|^{2}}d\mu(x)d\mu(y)  \right ) \frac{1}{r^{3/2}}dr \geq 0.
$$
Further, the double inner integral is  positive whenever $\mu$ is not the zero measure, implying that the final result is a positive number, which is the key argument in order to verify that $I_{1/2}$ is an inner product, thus reobtaining the main result of \cite{lyons2013} by a complete different argument. More generally, we have the following result. 

\begin{thm}\label{ber+cond} Let $\psi:[0, \infty) \to \mathbb{R}$ be a Bersntein function and $\gamma: X \times X \to [0, \infty)$ be a continuous CND kernel such that $x \to \gamma(x,x)$ is a bounded function. Consider the vector space 
$$
\mathfrak{M}_{1}(X; \gamma, \psi):= \{ \eta \in \mathfrak{M}(X), \quad   \psi(\gamma(x,y))\in L^{1}(|\eta|\times |\eta|) \text{ and } \eta(X)=0 \},
$$
then the function
$$
(\mu, \nu ) \in  \mathfrak{M}_{1}(X; \gamma, \psi) \times \mathfrak{M}_{1}(X; \gamma, \psi) \to I(\mu, \nu)_{\gamma, \psi}:=-\int_{X} \int_{X} \psi(\gamma(x,y))d\mu(x)d\nu(y)
$$
 defines an semi-inner product on  $\mathfrak{M}_{1}(X; \gamma, \psi)$. If $\psi$ is not a linear function and $2\gamma(x,y)=\gamma(x,x) + \gamma(y,y)$ only when $x=y$, then $I(\mu, \nu)_{\gamma, \psi}$ defines an inner product on  $\mathfrak{M}_{1}(X; \gamma, \psi)$.
\end{thm}

We emphasize that by Lemma \ref{unidimcase}   ($p$ is the constant $1$ function) $\psi(\gamma(x,y))\in L^{1}(|\eta|\times |\eta|)$ if and only if $x \to \psi(\gamma(x,z))\in L^{1}(|\eta|)$ for some (or every) $z \in X$.

For instance, if $X$ is a real Hilbert space $\mathcal{H}$, $\gamma(x,y)= \| x-y \|^{2}$ and $\psi(t)=t^{a/2}$, $0< a< 2$, then 
$$
(\mu, \nu ) \in  \mathfrak{M}_{1}(\mathcal{H}; t^{a/2}) \times  \mathfrak{M}_{1}(\mathcal{H}; t^{a/2}) \to I(\mu, \nu)_{ a/2}:=-\int_{\mathcal{H}} \int_{\mathcal{H}}  \|x-y\|^{a}d\mu(x)d\nu(y)
$$
defines an inner product on
$$
\mathfrak{M}_{1}(\mathcal{H}; t^{a/2}):= \{ \eta \in \mathfrak{M}(\mathcal{H}), \quad    \|x\|^{a} \in L^{1}(|\eta|) \text{ and } \eta(X)=0 \}.
$$

It is relevant to say that usually the inner product on Theorem \ref{ber+cond} is not complete (hence, $\mathfrak{M}_{1}(X; \gamma, \psi)$ is not a Hilbert space). For instance, on \cite{simon2018kernel} it is proved that the Gaussian kernel can be used to define an inner product on the space of tempered distributions on Euclidean spaces.

Another example occurs on the generalized real hyperbolic space. Let $\mathcal{H}$ be a Hilbert space and  define $\mathbb{H}:=\{(x,t_{x}) \in \mathcal{H}\times (0,\infty), \quad   t_{x}^{2} - \|x\|^{2} =1 \}$ be the real hyperbolic space relative to $\mathcal{H}$ and consider the kernel
$$
((x,t_{x}),(y,t_{y})) \in \mathbb{H} \times \mathbb{H} \to [(x,t_{x}),(y,t_{y})]:= t_{x}t_{y} - \langle x,y\rangle \in [1,\infty),
$$
which satisfies the relation  
$$
\cosh(d_{\mathbb{H}}((x,t_{x}),(y,t_{y}))) = [(x,t_{x}),(y,t_{y})],
$$
where $d_{\mathbb{H}}$ is a metric in $\mathbb{H}$. On \cite{farauthyper} or chapter $5$ in \cite{berg0}, it is proved that the metric $d_{\mathbb{H}}$  on $\mathbb{H}$ is a CND kernel, being so we can apply Theorem \ref{ber+cond} for the kernel $\gamma=d_{\mathbb{H}}$ and $\psi=t^{a/2}$, $0< a < 2$, then
$$
(\mu, \nu ) \in  \mathfrak{M}_{1}(\mathbb{H}; t^{a/2}) \times  \mathfrak{M}_{1}(\mathbb{H}; t^{a/2}) \to H(\mu, \nu)_{ a/2}:=-\int_{\mathbb{H}} \int_{\mathbb{H}}  d_{\mathbb{H}}(x,y)^{a/2}d\mu(x)d\nu(y)
$$
defines a inner product on
\begin{align*}
\mathfrak{M}_{1}(\mathbb{H}; t^{a/2}):= \{ \eta \in \mathfrak{M}(\mathbb{H}),& \quad   x \in \mathbb{H} \to d_{\mathbb{H}}(x,z)^{a/2} \in L^{1}(|\eta|)\\
&\text{for some (or every) }z \in \mathbb{H} \text{ and } \eta(\mathbb{H})=0 \}.
\end{align*}

We can also include the case $a=2$. A proof when $\mathbb{H}$ is finite dimensional was provided on \cite{lyons2014} using geometric properties of  hyperbolic spaces. Our proof relies on a Laurent type of approximation for the function $\arccosh(t)$. 

\begin{thm}\label{hyperinnprod} Let $\mathbb{H}$ be a real hyperbolic space, and consider the vector space
\begin{align*}
\mathfrak{M}_{1}(\mathbb{H}; t):= \{ \eta \in \mathfrak{M}(\mathbb{H}),& \quad   x \in \mathbb{H} \to d_{\mathbb{H}}(x,z) \in L^{1}(|\eta|)\\
&\text{for some (or every) }z \in \mathbb{H} \text{ and } \eta(\mathbb{H})=0 \}.
\end{align*}
Then
$$
(\mu, \nu ) \in  \mathfrak{M}_{1}(\mathbb{H}; t)\times  \mathfrak{M}_{1}(\mathbb{H}; t)\to H(\mu, \nu)_{ 1}:=-\int_{\mathbb{H}} \int_{\mathbb{H}}  d_{\mathbb{H}}(x,y) d\mu(x)d\nu(y)
$$
is an inner product.
\end{thm}

A different behaviour occurs on the generalized real spheres. Let $\mathcal{H}$ be a Hilbert space and  define $S^{\mathcal{H}}:=\{x \in \mathcal{H}, \quad    \|x\|=1 \}$ be the real sphere relative to $\mathcal{H}$. The kernel $d_{S^{\mathcal{H}}}$ defined on $S^{\mathcal{H}}$ by the relation  
$$
\cos (d_{S{\mathcal{H}}}(x,y))= \langle x,y \rangle_{\mathcal{H}}, \quad  x,y \in \mathcal{H}
$$
is a metric and defines a CND kernel as shown on \cite{gangolli}. However, unlikely the Hilbert space and the real hyperbolic space,   $d_{S^{\mathcal{H}}}$ is not a metric space of strong negative type, \cite{lyons2020strong}. Gangolli also proved on \cite{gangolli} that the metric on the other compact two-point homogeneous spaces  (real/complex/quaternionic projective spaces and the Cayley projective plane)  does not define a CND kernel.

The following Corollary of Theorem \ref{ber+cond}, connects the setting of metric spaces of strong negative type and the kernels on Theorem \ref{ber+cond}. 

\begin{cor}\label{ber+condcor} Let $\psi:[0, \infty) \to \mathbb{R}$ be a nonzero Bernstein function such that $\psi(0)=0$, \\ $\lim_{t \to \infty} \psi(t)/t =0$ and $(X, \gamma)$ is a metric space of negative type. Then, 
$$
(x,y) \in X \times X  \to D_{\psi, \gamma}(x,y):=\psi(\gamma(x,y))
$$
is a metric on $X$ and  $(X, D_{\psi, \gamma})$ is a metric space of strong negative type homeomorphic to $(X, \gamma)$.
\end{cor}

As an example of Corollary \ref{ber+condcor}, the Bersntein function $\psi(t)= \log (t+1)$, satisfies $\psi(0)=0$ and $\lim_{t \to \infty}\psi(t)/t =0$. In particular, on a Hilbert space $\mathcal{H}$ , $\log (\|x-y\| + 1 )$ is a metric on $\mathcal{H}$ that is homeomorphic with the Hilbertian topology and  this metric is of strong negative type. Interestingly we can apply  Corollary \ref{ber+condcor} again  in order to obtain that the same occurs with the metric $\log (\log (\|x-y\| + 1 ) + 1 )$. 

Returning to the kernel $(x,y) \in \mathcal{H}\times \mathcal{H} \to \|x-y\|^{a}$, we may ask ourselves what occurs when $a \geq 2$. 
The case $a=2$ is simpler, because
$$
-\int_{\mathcal{H}} \int_{\mathcal{H}}  \|x-y\|^{2}d\mu(x)d\nu(y)= 2\int_{\mathcal{H}} \int_{\mathcal{H}}  \langle x,y \rangle_{\mathcal{H}} d\mu(x)d\nu(y),
$$
for every $ \mu, \nu \in \mathfrak{M}_{1}(\mathcal{H}; t):=\{ \eta \in \mathfrak{M}(\mathcal{H}), \quad    \|x\|^{2} \in L^{1}(|\eta|) \text{ and } \eta(X)=0 \}$. This still defines a semi-inner product on  $\mathfrak{M}_{1}(\mathcal{H}; t)$, but the vector space
$$
\mathfrak{M}_{2}(\mathcal{H}; t) :=\{ \eta \in \mathfrak{M}_{1}(\mathcal{H}; t), \quad   \int_{\mathcal{H}}   \langle x,y \rangle_{\mathcal{H}}d\eta(x)=0, \text{ for every } y \in \mathcal{H}  \} \subset \mathfrak{M}_{1}(\mathcal{H}; t) 
$$
is equivalent to the zero measure on this inner product.  For an arbitrary measure $\eta \in \mathfrak{M}(\mathcal{H})$ such that $\|x\|^{2} \in L^{1}(|\eta|)$, the linear functional
$$
y \in \mathcal{H} \to   \int_{\mathcal{H}}   \langle x,y \rangle d\eta(x) \in \mathbb{R}
$$
is continuous, so there exists a vector $v_{\eta}$, which we call the vector mean of $\eta$, which represents the above continuous linear functional.

On the  case $a >2$, a different behaviour emerges. The double integral kernel does not define a semi-inner product on $\mathfrak{M}_{1}(\mathcal{H}, t^{a/2})$, however, if we restrict ourselves to the vector space space 
$$
\mathfrak{M}_{2}(\mathcal{H}; t^{a/2}) :=\{ \eta \in \mathfrak{M}(\mathcal{H}), \quad \|x\|^{a} \in L^{1}(|\eta|),  \eta(\mathcal{H})=0, v_{\eta}=0 \}
$$
for $2 < a < 4$ and using the representation given on Equation \ref{compleelltimes2} for the $CM_{2}$ function
$$
t^{a/2} = \frac{a(a -2)}{4\Gamma(2-a/2)}\int_{(0, \infty)}(e^{-rt} -1 + rt) \frac{1}{r^{a/2 +1}}dr,
$$
 by Fubini-Tonelli we obtain that if $\mu , \nu \in \mathfrak{M}_{2}(\mathcal{H}; t^{a/2})$
\begin{align*}
  \int_{\mathcal{H}}\int_{\mathcal{H}} \| x-y\|^{a}&d\mu(x)d\nu(y)\\
  &= \frac{a(a -2)}{4\Gamma(2-a/2)}\int_{(0, \infty)} \left (  \int_{\mathcal{H}}\int_{\mathcal{H}} e^{-r\| x-y\|^{2}}d\mu(x)d\nu(y)  \right ) \frac{1}{r^{a+1}}dr \geq 0.
\end{align*}

In particular, we can use the kernel $\| x-y\|^{a}$, $2<a < 4$, in order to define a metric on the space of Radon  probability measures on $\mathcal{H}$ with finite second moment,  but with a fixed vector mean.  

More generally, we have that.

 \begin{thm}\label{principal} Let $\ell \in \mathbb{N}$, $\psi: [0, \infty) \to \mathbb{R}$ be a continuous function on $CM_{\ell}$ and $\gamma: X \times X \to [0, \infty)$ be a continuous CND kernel such that $x \to \gamma(x,x)$ is a constant function. Consider the vector space 
\begin{align*}
\mathfrak{M}_{\ell}(X; \gamma, \psi):&= \{ \eta \in \mathfrak{M}(X), \quad   \psi(\gamma(x,y))\in L^{1}(|\eta|\times |\eta|) , \quad   \gamma(x,y)^{\ell}  \in L^{1}(|\eta|\times |\eta|)  \text{ and } \\
& \eta(X)=0, \quad  \int_{X}\int_{X}K_{-\gamma}(x,y)^{j}d\eta(x)d\eta(y)=0, \quad 1\leq j \leq \ell -1\}
\end{align*}
where $K_{-\gamma}$ is the  kernel in Theorem \ref{pontriequi},   then the function
$$
(\mu, \nu ) \in  \mathfrak{M}_{\ell}(X; \gamma, \psi)\times \mathfrak{M}_{\ell}(X; \gamma, \psi) \to I(\mu, \nu)_{\gamma, \psi}:=\int_{X} \int_{X} \psi(\gamma(x,y))d\mu(x)d\nu(y)
$$
 defines an semi-inner product on  $\mathfrak{M}_{\ell}(X; \gamma, \psi)$. If $\psi$ is not a polynomial of degree $\ell$ or less and  $2\gamma(x,y)=\gamma(x,x) + \gamma(y,y)$ only when $x=y$, then $I(\mu, \nu)_{\gamma, \psi}$ defines an inner product on  $\mathfrak{M}_{\ell}(X; \gamma, \psi)$. 
\end{thm}

From Equation  \ref{compleelltimes2} and the fact that $(-1)^{\ell}[e^{-tr} - \omega_{\infty, \ell}(rt)] \geq 0$, for every $t,r \geq 0$ and $\ell \in\mathbb{N}$ (this can be easily proved by induction on $\ell$), if $\psi \in CM_{\ell}$ on Theorem  \ref{principal} also belongs $C^{\ell-1}[0, \infty)$, we may lower the requirement $\gamma^{\ell} \in L^{1}(|\eta|\times |\eta|)$ to $\gamma^{\ell-1} \in L^{1}(|\eta|\times |\eta|)$ on the definition of  $\mathfrak{M}_{\ell}(X; \gamma, \psi)$.  

The fact that  we required additional properties on the function $\gamma(x,x)$  on Theorem \ref{principal}  compared to Theorem \ref{ber+cond}, is  related to the fact that the integrals  $\int_{X}\int_{X}\gamma(x,y)^{j}d\eta(x)d\eta(y)$ are difficult to analyse on the general setting of  Theorem \ref{ber+cond}. However, if  $\psi \in CM_{\ell}$ on Theorem  \ref{principal} also belongs $C^{\ell-1}[0, \infty)$ and all of its derivatives up to $\ell -1$ are  zero at the point $0$, then there is no polynomial part on Equation \ref{compleelltimes2}, and on this case we may only assume that $x \to \gamma(x,x)$ is a bounded function on Theorem  \ref{principal}. This is the case for the function $(-1)^{\ell}t^{a}$, $2(\ell -1) < a < 2\ell$.

As an example of Theorem \ref{principal}, if $X$ is a Hilbert space $\mathcal{H}$, $\gamma(x,y)= \| x-y \|^{2}$ and $\psi(t)=(-1)^{\ell}t^{a/2}$, $2(\ell-1)<a< 2\ell$, $\ell \in \mathbb{N}$ then 
$$
(\mu, \nu ) \in  \mathfrak{M}_{\ell}(\mathcal{H}; t^{a/2}) \times  \mathfrak{M}_{\ell}(\mathcal{H}; t^{a/2}) \to I(\mu, \nu)_{a/2}:=\int_{\mathcal{H}} \int_{\mathcal{H}}  (-1)^{\ell}\|x-y\|^{a}d\mu(x)d\nu(y)
$$
defines a inner product on the vector space
\begin{align*}
 \mathfrak{M}_{\ell}(\mathcal{H}; t^{a/2}):&= \{ \mu \in \mathfrak{M}(\mathcal{H}), \quad    \|x\|^{a} \in L^{1}(|\mu|), \quad \mu(\mathcal{H})=0, \text{ and } \\
 &  \quad \int_{\mathcal{H}}   \langle x,y_{1} \rangle \ldots \langle x,y_{j} \rangle d\mu(x)=0, \quad y_{1}, \ldots , y_{j} \in \mathcal{H} \quad 1\leq j \leq \ell-1 \}.
\end{align*}

Theorem  \ref{ber+cond}   and Theorem  \ref{principal}   on the case where $X$ is an Euclidean space $\mathbb{R}^{m}$  and $\gamma(x,y)= \|x-y\|^{2}$ were proved on \cite{mattner1997strict}.

\section{Space of functions defined by  derivatives of completely monotone functions}\label{Equimeasurability for derivatives of completely monotone functions}
As mentioned in \cite{lyons2013}, the fact that the energy distance defines a metric on a separable Hilbert space can be proved using the proposed method, but also follows as a consequence of the fact that if $\mathcal{H}$ is a separable Hilbert space,  then a measure $\mu \in \mathfrak{M}(\mathcal{H})$    such that $\|x\|^{a} \in L^{1}(|\mu|)$, $  a \in (0, \infty) \setminus 2\mathbb{N}$, satisfies
\begin{equation}\label{energydistance}
\int_{\mathcal{H}}\|x-y\|^{a}d\mu(x)=0 , \quad y \in \mathcal{H}
\end{equation}
if and only if $\mu$ is the zero measure, proved in   \cite{linde1986}, \cite{Koldobskii1987}.

On \cite{gaussinfi} it is proved that if $\psi \in CM_{0}$ and is not a constant function, then 
$$
\int_{\mathcal{H}}\psi(\|x-y\|^{2})d\mu(x)=0 , \quad y \in \mathcal{H}
$$
if and only if $\mu$ is the zero measure. In this section we prove  similar results on a much broader setting,  as a consequence of the results presented on Section \ref{Inner products defined by CND kernels and derivatives of completely monotone functions}. 

\begin{thm}\label{kobolnewneural}Let $\mathcal{H}$ be an infinite dimensional Hilbert space, $\ell \in \mathbb{Z}_{+}$ and $\phi ,\varphi  \in CM_{\ell}$.  If a measure  $ \mu \in \mathfrak{M}(\mathcal{H})$  such that $\|x\|^{2\ell} \in L^{1}(|\mu|)$ satisfies 
$$
 \int_{\mathcal{H}}\psi(\|x-y\|^{2})  d\mu(x)=0 \quad y \in \mathcal{H},
$$
where $\psi:=\phi - \varphi$ then it must hold that
$$
 \int_{\mathcal{H}}\psi(\|x-y\|^{2} +c)  d\mu(x)=0 \quad y \in \mathcal{H}, c\geq 0.
$$
In addition,  (even if $\mathcal{H}$ is not infinite dimensional),  $\psi$ is not a polynomial if and only if the only measure  $ \mu \in \mathfrak{M}(\mathcal{H})$  such that $\|x\|^{2\ell} \in L^{1}(|\mu|)$ satisfies 
$$
 \int_{\mathcal{H}}\psi(\|x-y\|^{2}+c)  d\mu(x)=0 \quad y \in \mathcal{H}, c\geq 0
$$ 
is the zero measure. 
\end{thm}

For some functions we can provide a version of Theorem \ref{kobolnewneural} on finite dimensional spaces.

\begin{lem}\label{kobolnew2}
Let $\ell \in \mathbb{N}$ and $\mathcal{H}$ be a Hilbert space. A measure  $ \mu \in \mathfrak{M}(\mathcal{H})$  such that $\|x\|^{2(\ell -1)} \in L^{1}(|\mu|)$ and $\psi(\|x-y\|^{2}) \in L^{1}(|\mu|\times |\mu|)$, satisfies 
$$
 \int_{\mathcal{H}}\psi(\|x-y\|^{2})  d\mu(x)=0 \quad y \in \mathcal{H}
$$
when $\psi: [0, \infty) \to \mathbb{R}$ is one of the following functions:
\begin{enumerate}
 \item[$(i)$] $\psi(t)= t^{a/2}$, \quad  $2 (\ell-1) < a < 2\ell $;
 \item[$(ii)$]  $\psi(t)= t^{\ell-1}\log(t)$, \quad  $\ell >1$;
 \item[$(iii)$] $\ell=1$ and  $\psi \in CM_{\ell}$  is not a polynomial, $\psi(0)\leq 0$.
\item[$(iv)$] $\ell=2$ and   $\psi \in CM_{\ell}$  is not a polynomial, $\psi(0)\leq 0$ but $\|x\|^{2\ell} \in L^{1}(|\mu|)$.
\end{enumerate}
if and only if $\mu$ is the zero measure. 
\end{lem}

We remark that on the case $(iv)$ we may  withdraw the additional assumption $\|x\|^{2\ell} \in L^{1}(|\mu|)$ if $\psi \in C^{\ell-1}[0,\infty)$.

\section{Proofs}\label{Proofs}

\subsection{\textbf{Section \ref{Conditionally positive definite kernels}}}


	\begin{proof}[\textbf{Proof of Theorem \ref{gensupport}}]The converse is immediate.\\
	Suppose that $\gamma$ is $P$-CPD. Since $P$ is finite dimensional there exists a basis  $p_{1}, \ldots , p_{m} \in P$ for it such that $p_{i}(\xi_{j})= \delta_{i,j}$.   By the integrability assumptions on the functions $p_{i}$ and $\gamma(x, \xi_{j})$, the kernel $K_{\gamma} \in L^{1}(|\mu|\times |\mu|)$, and 
	$$
	\int_{X}\int_{X}\gamma(x,y)d\mu(x)d\overline{\mu}(y)= \int_{X}\int_{X}K_{\gamma}(x,y)d\mu(x)d\overline{\mu}(y)
	$$
	the conclusion will follow from Lemma \ref{extmmddominio}.   
\end{proof}

\begin{proof}[\textbf{Proof of Lemma \ref{unidimcase}} ]
Let $\mu \in 	\mathfrak{M}_{P}(X)$  for which $\gamma(x,y) \in L^{1}(|\mu|\times |\mu|)$. Let \\ $A:=\{\xi \in X, \quad \gamma( \cdot , \xi)  \in L^{1}(|\mu|)\}$, which by   Fubini-Tonelli its complement has $|\mu|$ zero measure. If $A \cap\{ \xi, \quad p(\xi) \neq0\} \neq \emptyset$, the result is a consequence of Theorem \ref{gensupport}. On the other hand, if  $A \cap\{ \xi, \quad p(\xi) \neq0\} = \emptyset$, note that the kernel $\gamma$ is positive definite when restricted to the closed set  $B:=\{ \xi, \quad p(\xi) = 0\}$, $A \subset B$,  and 
$$
\int_{X}\int_{X}\gamma(x,y)d\mu(x)d\overline{\mu}(y)= \int_{B}\int_{B}\gamma(x,y)d\mu(x)d\overline{\mu}(y).
$$
The conclusion follows from Lemma \ref{extmmddominio}.\\
Now, under the additional requirements on $p$ and $\gamma$ it is easy to see that $\gamma$ is $[p]$-PD if and only if the kernel
$$
(x,y) \in X \times X \to \beta(x,y):=\frac{\gamma(x,y)}{p(x)p(y)} \in \mathbb{R} 
$$
is CPD. Note that is sufficient to prove the $3$ equivalences on the kernel $\beta$ for any measure $\eta \in \mathfrak{M}(X)$ with $\eta(X)=0$, because we can take $d\eta=pd\mu$.\\
The kernel $d(x,y):=(-2\beta(x,y) + \beta(x,x) + \beta(y,y))^{1/2}$ is a pseudometric on $X$, because $d^{2}$ is a CND kernel with $d(x,x)=0$ for all $x \in X$, so it   satisfies the triangle inequality. Since the function $\beta(x,x)$ is bounded, the relations
$$
\beta \in L^{1}(|\eta|\times |\eta|), \quad \beta(x,z) \in  L^{1}(|\eta|) \text{ for some } z \in X, \quad  \beta(x,z) \in  L^{1}(|\eta|) \text{ for every } z \in X
$$
are respectively equivalent to the relations 
$$
d \in L^{2}(|\eta|\times |\eta|), \quad d(x,z) \in  L^{2}(|\eta|) \text{ for some } z \in X, \quad  d(x,z) \in  L^{2}(|\eta|) \text{ for every } z \in X.
$$
The conclusion that these $3$ properties are equivalent for the kernel $d$ follows directly from the triangle inequality.
\end{proof}

\begin{proof}[\textbf{Proof of Lemma \ref{extmmddominio}}] Assume without loss of generalization that $\mu$ is a nonnegative measure. The fact that the set $X_{\mu}$ satisfies $\mu(X - X_{\mu})=0$ is a direct consequence of the Fubini-Tonelli Theorem.\\
Also, by the  Radon hypothesis, there exists a sequence of nested compact sets $(\mathcal{C}_{n})_{n \in \mathbb{N}}$  for which  $\mu(X - \cup_{n \in \mathbb{N}} \mathcal{C}_{n})=0$. In particular, by the Dominated Convergence Theorem, the $L^{1}(\mu \times \mu)$ convergence holds 
$$
\int_{X}\int_{X}K(x,y)\chi_{\mathcal{C}_{n}}(x)\chi_{\mathcal{C}_{n}}(y)d\mu(x)d\mu(y) \to \int_{X}\int_{X}K(x,y)d\mu(x)d\mu(y),
$$
because $\mu\times \mu(X \times [X-\cup_{n \in \mathbb{N}} \mathcal{C}_{n}])=0$. The function $\sqrt{K(x,x)} \in L^{1}(\chi_{\mathcal{C}_{n}}\mu)$, so by Lemma \ref{initialextmmddominio}, $K_{\mu^{n}} \in \mathcal{H}_{K}$, where $\mu^{n}:=\chi_{\mathcal{C}_{n}}d\mu$, and
\begin{align*}  
\langle K_{\mu^{n}} -K_{\mu^{m}},& K_{\mu^{n}} - K_{\mu^{m}} \rangle_{\mathcal{H}_{K}}\\
&= \int_{X}\int_{X}K(x,y)[\chi_{\mathcal{C}_{n}}(x) -\chi_{\mathcal{C}_{m}}(x)][\chi_{\mathcal{C}_{n}}(y) - \chi_{\mathcal{C}_{m}}(y)]d\mu(x)d\mu(y) \xrightarrow[m,n \to \infty]{}  0 
\end{align*}  
which proves that the sequence $(K_{\mu^{n}})_{n \in \mathbb{N}}$ is Cauchy, in particular, convergent to an element $K_{\mu} \in \mathcal{H}_{K}$. Since  $\mathcal{H}_{K}$ is a RKHS, convergence  in norm implies pointwise convergence, so 
$$
K_{\mu}(z) = \lim_{n \to \infty} K_{\mu}^{n}(z)=  \lim_{n \to \infty}\int_{\mathcal{C}_{n}}K(x,z)d\mu(x)= \int_{X}K(x,z)d\mu(x),
$$
for every $z \in X_{\mu}$, which proves our claim. \\
Now, if $K \in L^{1}(\mu\times \eta)$, we have that
$$
\langle K_{\eta^{n}}, K_{\mu^{n}}\rangle_{\mathcal{H}_{K}}= \int_{X} \int_{X}k(x,y)\chi_{\mathcal{D}_{n}}(x)\chi_{\mathcal{C}_{n}}(y)d\eta(x)d\mu(y).
$$ 
The left hand side of this equality converge to $\langle K_{\eta}, K_{\mu}\rangle_{\mathcal{H}_{K}}$, while the right hand side converge to $\int_{X} \int_{X}k(x,y)d\eta(x)d\mu(y)$ by the Dominated Convergence Theorem.
\end{proof}	

\subsection{\textbf{Section \ref{Inner products defined by CND kernels and derivatives of completely monotone functions} } }
Throughout the rest of the paper, we use the well known fact that a Hermitian kernel $\gamma:X \times X \to \mathbb{C}$ is CND if and only if the kernel $e^{-r\gamma(x,y)}$ is positive definite for every $r>0$, page $74$ in \cite{berg0}.

Next Lemma is an improvement of Lemma \ref{unidimcase} for $p$ as the set of  constant functions. We use CND instead of CPD because it is how we apply this result.

\begin{lem}\label{estimativa} Let   $\gamma: X \times X \to \mathbb{R}$ be a continuous CND kernel such that $\gamma(x,x)$ is a bounded function, $\mu \in \mathfrak{M}(X)$ and $\theta >0$. Then, the following assertions are equivalent
\begin{enumerate}
    \item[$(i)$] $\gamma \in L^{\theta}(|\mu|\times |\mu|)$;
    \item[$(ii)$] The function $x \in X \to \gamma(x,z) \in L^{\theta}(|\mu|)$ for some $z \in X$;
    \item[$(iii)$] The function $x \in X \to \gamma(x,z) \in L^{\theta}(|\mu|)$ for every $z \in X$.
    \end{enumerate}
\end{lem}

\begin{proof}Since $\gamma$ is CND there exists a CND kernel $\beta: X \times X \to \mathbb{R}$, for which $\beta(x,x)=0$ for every $x \in X$, $\beta^{1/2}$ is a pseudometric on $X$ and $\gamma(x,y)= \beta(x,y) + \gamma(x,x)/2 + \gamma(y,y)/2$.\\
Since $\gamma(x,x)$ is bounded and $\mu$ is a finite measure, for $\theta \geq 1$ the three equivalences for $\gamma$ are respectively equivalent to the three equivalences for the CND kernel $\beta$ by the Minkowsky inequality. If $\theta \in (0, 1)$ the same relation occurs, but it follows from the general relation on $L^{\theta}$ spaces 
$$
\int |f+g|^{\theta}  \leq \int |f|^{\theta} + \int |g|^{\theta}. 
$$
In particular, we may suppose that $\gamma$ is a CND kernel for which $\gamma(x,x)=0$ for every $x \in X$.\\
If $\gamma(x,y)^{\theta} \in L^{1}(|\mu|\times |\mu|)$ then there exists  $z\in X$ for which  $\gamma(x,z)^{\theta}  \in L^{1}(|\mu|)$ by the Fubini-Tonelli Theorem.\\
If $x \in X \to \gamma(x,z)^{\theta}  \in L^{1}(|\mu|)$ for some $z \in X$, then for every $y \in X$ 
$$
(\gamma(x,y))^{\theta}=((\gamma(x,y))^{1/2})^{2\theta} \leq  (\gamma(x,z)^{1/2} + \gamma(y,z)^{1/2})^{2\theta}.  
$$
For $\theta \geq 1/2$, the functions inside the parenthesis on the right hand side of the previous equation are elements of $L^{2\theta}(|\mu|)$ ($x$ variable), which by Minkowski Theorem we obtain the integrability of $ x \in X \to \gamma(x,y)^{\theta}$. Integrating the Minkowsky inequality with respect to $d|\mu|(y)$, we also obtain that $\gamma^{\theta} \in L^{1}(|\mu|\times |\mu|)$.
For $0< \theta < 1/2 $, the proof is the same but it follows from the from the general relation on $L^{2\theta}$ spaces as mentioned above.
\end{proof}

\begin{proof}[\textbf{Proof of Theorem \ref{ber+cond}}]
For the first claim it is sufficient to prove that $I(\mu, \mu)_{\gamma, \psi} \geq 0$ by the linearity of the integration involved. Indeed, by Equation \ref{compleelltimes2}, we have that
$$
-\psi(\gamma(x,y)) = a + b\gamma(x,y) + \int_{(0, \infty)}\frac{e^{-r\gamma(x,y)}-1}{r}d\lambda(r),
$$
where $ b \leq 0$ and $\lambda$ is a nonnegative Radon measure such that $\min \{1,r^{-1}\} \in L^{1}(\lambda)$. Consequently,  $-\psi(\gamma(x,y))$ is a CPD kernel, the conclusion is then  consequence of Corollary \ref{unidimcase}.\\
If $\psi$ is not a linear function, then $\lambda((0, \infty))>0$, because the representation on Equation \ref{compleelltimes2} is unique. \\
If $\mu \in \mathfrak{M}_{1}(X; \gamma, \psi)$, then  the $3$ functions that describes $\psi(\gamma(x,y))$ are in $L^{1}(|\mu|\times |\mu|)$, because $e^{-r\gamma(x,y)}-1 \leq 0$ for every $r >0$ and $x,y \in X$ and $b \leq 0$. Since
$$
(1-e^{-rt})\leq r(1+t)\min \{1, r^{-1}\}, \quad r,t \geq 0
$$
we can apply  Fubini-Tonelli and obtain that
\begin{align*}
-\int_{X}\int_{X}\psi(\gamma(x,y))d\mu(x)d\mu(y)=&  b\int_{X}\int_{X}\gamma(x,y)d\mu(x)d\mu(y)  \\
&+ \int_{(0, \infty)}\left [\int_{X}\int_{X}e^{-r\gamma(x,y)}d\mu(x)d\mu(y)\right ] \frac{1}{r}d\lambda(r).
\end{align*}
 The first double integral is non positive by Corollary \ref{unidimcase}. Since $2\gamma(x,y)= \gamma(x,x) + \gamma(y,y)$ only when $x=y$, the kernel $e^{-r\gamma(x,y)}$ is ISPD for every $r>0$ by Theorem $4.2$  in \cite{gaussinfi}, so 
$$
\int_{X}\int_{X}e^{-r\gamma(x,y)}d\mu(x)d\mu(y) >0, \quad r>0
$$
and the conclusion follows because $\lambda((0, \infty))>0$.

\end{proof}

\begin{proof}[\textbf{Proof of Theorem \ref{hyperinnprod}}]By equation $4.38.2$ in \cite{NIST:DLMF}, we have that for $t \geq 1$
$$
\arccosh (t)= \log (2) + \log(t) -\sum_{k=1}^{\infty}\frac{(2k)!}{2^{2k}(k!)^{2}}\frac{t^{-2k}}{2k}.
$$
In \cite{berg0} it is proved that $\log([x,y])$ is a CND kernel on $\mathcal{H}$ while by \cite{gaussinfi} the positive definite kernel  $[x,y]^{-2k}$ on $\mathbb{H}$ is ISPD for every $k \in \mathbb{N}$. Since the series appearing on the $\arccosh$ formula above only contains nonnegative numbers, we may reverse the order the summation with integration for any $\eta \in \mathfrak{M}_{1}(\mathbb{H}; t)$. Consequently, if $\mu$ is not the zero measure
\begin{align*}
-\int_{\mathbb{H}} \int_{\mathbb{H}}  d_{\mathbb{H}}(x,y) d\mu(x)d\mu(y)&=-\int_{\mathbb{H}} \int_{\mathbb{H}}  \arccosh ([x,y]) d\mu(x)d\mu(y)\\
&=  \int_{\mathbb{H}} \int_{\mathbb{H}} -\log([x,y])+ \sum_{k=1}^{\infty}\frac{(2k)!}{2^{2k}(k!)^{2}2k}[x,y]^{-2k}d\mu(x)d\mu (y)\\
&\geq  \int_{\mathbb{H}} \int_{\mathbb{H}}\sum_{k=1}^{\infty}\frac{(2k)!}{2^{2k}(k!)^{2}2k}[x,y]^{-2k}d\mu(x)d\mu (y)\\
&= \sum_{k=1}^{\infty}\frac{(2k)!}{2^{2k}(k!)^{2}2k} \int_{\mathbb{H}} \int_{\mathbb{H}}[x,y]^{-2k}d\mu(x)d\mu (y) >0.
\end{align*}   
\end{proof}

\begin{proof}[\textbf{Proof of Corollary \ref{ber+condcor}}] By Remark $3.3$-(iv)  on  \cite{bers}, if $\psi$ satisfy these assumptions then we can write the kernel $D_{\psi, \gamma}$ as
$$
D_{\psi, \gamma}(x,y)= \psi(\gamma(x,y))= \int_{(0, \infty)}\frac{1-e^{-r\gamma(x,y)}}{r}d\lambda(r)
$$
where $\lambda$ is a nonnegative Radon measure such that $\min \{1,r^{-1}\} \in L^{1}(\lambda)$. Because $\gamma$ is a metric, we have that 
$$
1-e^{-r\gamma(x,y)} \leq [1-e^{-r\gamma(x,z)}] + [1-e^{-r\gamma(z,y)}], \quad x,y,z \in X, 
$$
Which proves that $D_{\psi, \gamma}(x,y) \leq D_{\psi, \gamma}(x,z) + D_{\psi, \gamma}(z,y)$. \\
The topologies are equivalent because $\psi$ is necessarily an increasing function with $\psi(0)=0$, so $\psi(t_{n}) \to 0$ if and only if $t_{n} \to 0$.\\
The metric space $(X, D_{\psi, \gamma})$ has strong negative type because the kernel $\gamma$ is continuous on the metric topology  $(X,\gamma)$, $\psi$ is not a linear function and the remaining requirements for Theorem \ref{ber+cond} are satisfied. 
\end{proof}

In order to prove the next result, we will use the same  infinite  dimensional multinomial theorem that was used to prove that the Gaussian kernel is ISPD on Hilbert spaces on \cite{gaussinfi}. If $\mathcal{H}$ is a real Hilbert space and $(e_{\xi})_{\xi \in \mathbb{I}}$ is a complete orthonormal basis for it, then for every $n \in \mathbb{N}$
\begin{equation}\label{multinomial}
\langle x,y\rangle^{n}= \left (\sum_{\xi \in \mathbb{I}} x_{\xi}y_{\xi} \right )^{n} = \sum_{\alpha \in (\mathbb{I},\mathbb{Z}_{+}), |\alpha| =n} \frac{n!}{\alpha!} x^{\alpha}y^{\alpha}
\end{equation}
where $x_{\xi}= \langle x, e_{\xi} \rangle $, $(\mathbb{I},\mathbb{Z}_{+})$ is the space of functions from $\mathbb{I}$ to $\mathbb{Z}_{+}$, the condition $|\alpha|=n$ means that $\sum_{\xi \in \mathbb{I}}\alpha(\xi)=n$ (in particular $\alpha$ must be  the zero function except for a finite number of points). Also $\alpha! = \prod_{\xi \in \mathbb{I}}\alpha(\xi)!$ (which makes sense because $0!=1$) and $x^{\alpha}= \prod_{ \alpha(\xi) \neq 0}x_{\xi}^{\alpha(\xi)}$. This result can be proved using approximations of $\langle x,y\rangle$ on finite dimensional spaces and the multinomial theorem on those spaces.  The number $ \lfloor l \rfloor$ stands for the smallest integer less then or equal to $l$.

On the next Lemma we use the fact that  for a continuous positive definite kernel $K: X \times X \to \mathbb{C}$ a  measure  $ \mu \in \mathfrak{M}_{\sqrt{K}}(X )$ satisfy
$$
\int_{X}K(x,y)d\mu (x)=0, \quad y \in X
$$
if and only if $\int_{X}\int_{X}K(x,y)d\mu (x)d\overline{\mu }(y)=0$, which can be seen on   \cite{micchelli2006universal}, \cite{Sriperumbudur3}.

\begin{lem}\label{hilbertcondk} Let $\mathcal{H}$ be a real Hilbert space, $n \in \mathbb{N}$ and $\mu \in \mathfrak{M}(\mathcal{H})$. Suppose that  $\|x-y\|^{2n} \in L^{1}(|\mu|\times |\mu|)$, then 
$$
\langle x,y\rangle^{k}\|x\|^{2i}\|y\|^{2j}  \in  L^{1}(|\mu|\times |\mu|), \quad k, i, j \in \mathbb{Z}_{+}, \quad  k+i+j\leq n. 
$$
Moreover, if $\int_{\mathcal{H}}\int_{\mathcal{H}} \langle x,y \rangle^{k}d\mu(x)d\mu(y)=0$ for every $0\leq k \leq n-1$, then
\begin{align*}
(-1)^{n}\int_{\mathcal{H}}\int_{\mathcal{H}}&\|x-y\|^{2n}d\mu(x)d\mu(y)\\
&= \sum_{l=0}^{ \lfloor n/2 \rfloor} \binom{n}{2l}\binom{2l}{l}2^{n-2l}\int_{\mathcal{H}}\int_{\mathcal{H}} \langle x,y\rangle^{n-2l}\|x\|^{2l}\|y\|^{2l}d\mu(x)d\mu(y) \geq 0,
\end{align*}
and
$$
\int_{\mathcal{H}}\int_{\mathcal{H}}\|x-y\|^{2m}d\mu(x)d\mu(y)=0, \quad 0 \leq m \leq n-1.
$$
\end{lem}

\begin{proof}By Lemma \ref{estimativa}, the fact that  $\|x-y\|^{2n} \in L^{1}(|\mu|\times |\mu|)$ is equivalent at  $\|x\|^{2n} \in L^{1}(|\mu|)$. Since $|\langle x,y\rangle^{k}\|x\|^{2i}\|y\|^{2j}| \leq \|x\|^{2i+k}\|y\|^{2j+k}$
 and $\|x\|^{2i+k}\leq \max\{1,\|x\|^{2n} \}$, we obtain the desired integrability.\\
 Note that
 $$
 \|x-y\|^{2m} = (\|x\|^{2} + \|y\|^{2} -2\langle x,y\rangle  )^{m}= \sum_{k=0}^{m} \sum_{i=0}^{m-k}\binom{m}{k}\binom{m-k}{i} (-2)^{k}\langle x,y\rangle^{k}\|x\|^{2i}\|y\|^{2(m-k-i)}
 $$
If $k + 2i \leq n-1$, then by the hypothesis
\begin{align*}
0&=\int_{\mathcal{H}}\int_{\mathcal{H}}\langle x,y\rangle^{k +2i}d\mu(x)d\mu(y)= \int_{\mathcal{H}}\int_{\mathcal{H}}\langle x,y\rangle^{k} \left ( \sum_{\xi \in \mathbb{I}}x_{\xi}y_{\xi}  \right )^{2i}  d\mu(x)d\mu(y)\\
&= \int_{\mathcal{H}}\int_{\mathcal{H}}\langle x,y\rangle^{k} \left ( \sum_{\xi \in \mathbb{I}}x_{\xi}y_{\xi}  \right )^{2i}  d\mu(x)d\mu(y)\\
&=\int_{\mathcal{H}}\int_{\mathcal{H}}\langle x,y\rangle^{k} \left ( \sum_{ |\alpha| =2i} \frac{2i!}{\alpha!} x^{\alpha}y^{\alpha} \right )  d\mu(x)d\mu(y)\\
&=\sum_{ |\alpha| =2i} \frac{2i!}{\alpha!}     \int_{\mathcal{H}}\int_{\mathcal{H}}\langle x,y\rangle^{k} x^{\alpha}y^{\alpha}  d\mu(x)d\mu(y).
\end{align*}
But then, $\int_{\mathcal{H}}\int_{\mathcal{H}}\langle x,y\rangle^{k} x^{\alpha}y^{\alpha}  d\mu(x)d\mu(y) =0$ for every  $\alpha \in (\mathbb{I}, \mathbb{Z}_{+})$ with  $|\alpha|=2i$,  because the kernel inside the double integral is positive definite,  continuous and satisfies the conditions on Lemma \ref{initialextmmddominio}. In particular,  since for every $y \in \mathcal{H}$ and $|\alpha|=2i$ there exists a sequence $(y_{l})_{l \in \mathbb{N}}$ that converges to $y$ and  $y_{l}^{\alpha}\neq 0$, we have that
$$
\int_{\mathcal{H}}\langle x,y\rangle^{k} x^{\alpha}  d\mu(x)=0, \quad y \in \mathcal{H}, \alpha \in (\mathbb{I}, \mathbb{Z}_{+}), |\alpha|=2i.
$$
Then
\begin{align*}
&\int_{\mathcal{H}}\int_{\mathcal{H}}\langle x,y\rangle^{k} \|x\|^{2i} \|y\|^{2(m-k-i)}d\mu(x)d\mu(y)\\
&=\sum_{ |\beta| =m-k-i} \sum_{ |\alpha| =i} \frac{(m-k-i)!}{\beta!}\frac{i!}{\alpha!}\int_{\mathcal{H}}\int_{\mathcal{H}}\langle x,y\rangle^{k} x^{2\alpha}y^{2\beta}  d\mu(x)d\mu(y)=0.
\end{align*}
By symmetry, the same double integral is zero when $k+2(m-k-i) \leq n-1$. Those two relations occur only when $n=m$ and   $2i = 2(n-i-k)$. The remaining terms on the sum  when $n=m$ are exactly those on the statement on the theorem after a simplification using those two equalities.  The conclusion follows because  the kernel $ \langle x,y\rangle^{k}\|x\|^{2l}\|y\|^{2l}$ is continuous, positive definite and satisfies the conditions on Lemma \ref{initialextmmddominio} 
\end{proof}

\begin{cor}\label{hilbertcondkcor}Let  $\gamma: X \times X \to [0, \infty)$ be a continuous CND kernel such that $x \to \gamma(x,x)$ is a constant function and $2\gamma(x,y)=\gamma(x,x) + \gamma(y,y)$ only when $x=y$. Then for  $n \in \mathbb{N}$ and $\mu \in \mathfrak{M}(X)$ such  that  $\gamma^{n} \in L^{1}(|\mu|\times |\mu|)$,  the kernel $K_{-\gamma}$  defined in  Theorem \ref{pontriequi} satisfies $(K_{-\gamma})^{m} \in L^{1}(|\mu|\times |\mu|)$, $0\leq m \leq n $ and if   
$$
\int_{X}\int_{X}K_{-\gamma}(x,y)^{m}d\mu(x)d\mu(y)=0, \quad 0 \leq m \leq n-1,
$$
then
$$
(-1)^{n}\int_{X}\int_{X}\gamma(x,y)^{n}d\mu(x)d\mu(y)\geq 0 
$$
and 
$$
\int_{X}\int_{X}\gamma(x,y)^{m}d\mu(x)d\mu(y)= 0, \quad 0\leq m \leq n-1. 
$$\end{cor}	
	
\begin{proof}By the hypothesis on $\gamma$, there exists a Hilbert space $\mathcal{H}$ and a continuous and injective function $T: X \to \mathcal{H}$, such that $\gamma(x,y)= \|T(x) - T(y)\|_{\mathcal{H}}^{2}+ c$, where $c \geq 0$ is the value of $\gamma$ on the diagonal. If $\mu \in \mathfrak{M}(X)$ is a measure satisfying the conditions on the Corollary, then the image measure $\mu_{T} \in \mathfrak{M}(\mathcal{H})$ satisfies the same conditions of Lemma \ref{hilbertcondk}. The conclusion follows by standard properties of image measures.\end{proof}	
	
\begin{proof}[\textbf{Proof of Theorem \ref{principal}}] By Equation \ref{compleelltimes}, we have that
$$
\psi(\gamma(x,y))=\int_{(0,\infty)} \frac{e^{-\gamma(x,y)r} - e_{\ell}(r)\omega_{\ell,\infty}(\gamma(x,y)r)}{r^{\ell}} d\lambda(r) + \sum_{k=0}^{\ell}a_{k}\gamma(x,y)^{k}.
$$
By the hypothesis, the $\ell +2$ functions above are in $L^{1}(|\mu| \times |\mu|)$. Corollary \ref{hilbertcondkcor} implies that
$$
 \int_{X}\int_{X}\sum_{k=0}^{\ell}a_{k}\gamma(x,y)^{k}d\mu(x)d\mu(y) =  \int_{X}\int_{X}a_{\ell}\gamma(x,y)^{\ell}d\mu(x)d\mu(y) \geq 0.
$$
On the other hand,  because of Lemma \ref{change} we can apply  Fubini-Tonelli, and then
\begin{align*}
 &\int_{X}\int_{X} \left [\int_{(0,\infty)} \frac{e^{-\gamma(x,y)r} - e_{\ell}(r)\omega_{\ell,\infty}(\gamma(x,y)r)}{r^{\ell}} d\lambda(r)\right ]d\mu(x)d\mu(y) \\
 =&\int_{(0,\infty)}\frac{1}{r^{\ell}}\left [\int_{X}\int_{X} e^{-\gamma(x,y)r} d\mu(x)d\mu(y)\right ]d\lambda(r) \geq 0,
\end{align*}
because the inner double  integral is a nonnegative number for every $r>0$ by \cite{gaussinfi}.\\
Because the representation for $\psi$ is unique, if $\psi$ is not a polynomial of degree $\ell$ or less then $\lambda((0, \infty))>0$, also, if $2\gamma(x,y)=\gamma(x,x) + \gamma(y,y)$ only when $x=y$, by  \cite{gaussinfi} the inner double  integral is a positive number for every $r>0$ when $\mu$ is not the zero measure, and then the triple integral is a positive number as well.\end{proof}	

\begin{lem}\label{change}  There exists an $M>0$, which only depends on $\ell \in \mathbb{Z_{+}}$ for which
\begin{equation}\label{changeformula} 
|e^{-rt} - e_{\ell}(r)\omega_{\ell, \infty}(rt) | \leq Mr^{\ell}(1+t^{\ell})\min \{ 1,r^{-\ell}\}, \quad r > 0, t\geq 0.  
\end{equation}
\end{lem}

\begin{proof} Note that $r^{\ell} \min \{ 1,r^{-\ell} \}= \min \{ r^{\ell}, 1 \}$.\\ 
Case $r \geq 1$: On this case, the right hand side of Equation \ref{changeformula}  is $(1+t^{\ell})$, while  the left hand side is 
$$
|e^{-rt} - e_{\ell}(r)\omega_{\ell, \infty}(rt) | \leq 1 + |e_{\ell}(r)\omega_{\ell, \infty}(rt)| \leq 1 + \sum_{l=0}^{\ell-1}|e_{\ell}(r)r^{l}|t^{l}/ l!.
$$
Since each function $|e_{\ell}(r)r^{l}|$ is bounded, the results follows from the fact that $t^{l} \leq 1+t^{\ell}$.

Case $r <1$: On this case, the right hand side of Equation \ref{changeformula}  is $(1+t^{\ell})r^{\ell}$, while  the left hand side is
$$
|e^{-rt} - e_{\ell}(r)\omega_{\ell, \infty}(rt) |\leq  |e^{-rt} - \omega_{\ell, \infty}(rt)| + |(e_{\ell}(r)-1)\omega_{\ell, \infty}(rt)|.
$$
The function  $[e^{-s} - \omega_{\ell, \infty}(s)]/ s^{\ell}$ is a bounded function on $s \in [0, \infty)$, and from this we obtain the desired inequality for $|e^{-rt} - \omega_{\ell, \infty}(rt)|$.\\ 
On the other function we have that
$$
|(e_{\ell}(r)-1)\omega_{\ell, \infty}(rt)| \leq \sum_{l=0}^{\ell-1}|(e_{\ell}(r)-1)r^{l}|t^{l}/ l!.
$$
Similarly, since $e_{\ell}(r)-1= -e^{-r}\sum_{k=\ell}^{\infty}r^{k}/k!$ the functions  $(e_{\ell}(r)-1)r^{l}r^{-\ell}$ are bounded  on $r \in (0, 1)$  and from this we also obtain the desired inequality for $|(e_{\ell}(r)-1)\omega_{\ell, \infty}(rt)|$, which concludes the proof.
\end{proof}

\subsection{\textbf{Section \ref{Equimeasurability for derivatives of completely monotone functions}}}	

\begin{proof}[\textbf{Proof of Theorem \ref{kobolnewneural}}] Since $\mathcal{H}$ is infinite dimensional, take $(e_{\iota})_{\iota \in \mathbb{N}}$ be an orthonormal sequence of vectors in $\mathcal{H}$. By the Dominated Convergence Theorem, we have that
$$
0=\int_{\mathcal{H}}\psi(\|x-y -  re_{\iota}\|^{2})d\mu(x) \to \int_{\mathcal{H}}\psi(\|x-y \|^{2} + r^{2})d\mu(x), \quad y \in \mathcal{H} , \quad r \in \mathbb{R}
$$
because  $\langle x-y, e_{\iota}\rangle \to 0$ as $\iota \to \infty$ and $|\psi(t)|  \leq |\varphi(t)| + |\phi(t)|\lesssim (1+t)^{\ell}$, which proves the first assertion.\\
Now, if   $\psi$ is a polynomial of degree $n$, let $t_{1}, \ldots, t_{N} \in \mathbb{R}$, $c_{1}, \ldots, c_{N} \in \mathbb{R}$ (not all null) such that $ \sum_{i=1}^{N}c_{i}p(t_{i})=0$ for every $p \in \pi_{2n}(R)$. Then if $\|v\|=1$, the measure $\mu:= \sum_{i=1}^{N}c_{i}\delta(t_{i}v) \in \mathfrak{M}(\mathcal{H})$ is  nonzero and 
$$
\int_{\mathcal{H}}\psi(\|x-y\|^{2})d\mu(x)= \sum_{i=1}^{N}c_{i}\psi(\|y- \langle y,v \rangle v \|^{2} + ( \langle y,v \rangle -t_{i} )^{2})=0
$$
because this function is polynomial of degree $2n$ for every fixed $y \in \mathcal{H}$.\\
For the converse, first, we show that is sufficient to prove the case $\ell=0$.\\
Indeed,  the function $c \in (0, \infty)\to F(c):= \psi(\|x-y\|^{2}+c) \in \mathbb{R}$ is  differentiable for every $x,y \in \mathcal{H}$, and
$$
\frac{\partial F}{\partial c}(y) = \psi^{\prime}(c+ \|x-y\|^{2}).
$$
Since $\psi= \varphi - \phi$, and those functions are elements of $ CM_{\ell}$, we have that $|\psi^{\prime}(t+c)| \lesssim (1+t)^{\ell-1}$, for every $c>0$. In particular, the derivative  is a function in  $L^{1}(|\mu|)$ and
\begin{equation}\label{recursion2}
\int_{\mathcal{H}}\psi^{\prime }(c+\|x-y\|^{2})d\mu(x)=0, \quad y \in \mathcal{H}, \quad c>0. 
\end{equation}
Since $\psi^{\prime }(c +\cdot) $ also is the difference between two functions in $ CM_{\ell-1}$ for every $c>0$, by induction, we may assume that $\ell=0$.\\ 
Assume that $\psi$ is not a polynomial and $\mu$ is a nonzero measure that satisfy the equality on the statement of the Theorem. The function $\psi(c+ \cdot)$ is the difference between two completely monotone functions on $[0, \infty)$, so there exists a measure $\beta_{c}$ in $[0, \infty)$ for which 
$$
\psi(c+ t)= \int_{[0, \infty)}e^{-rt}d\beta_{c}(r), \quad c>0, t\geq0
$$
and $d\beta_{c+s}(r)= e^{-rs}d\beta_{c}(r)$ for every $c,s > 0$. Integrating the function on the hypotheses with respect to the measure $d\mu(y)$, we obtain that
\begin{equation}\label{nonposgauss}
0=\int_{\mathcal{H}}\int_{\mathcal{H}}\psi(c+\|x-y\|^{2})d\mu(x)d\mu(y) = \int_{[0, \infty)}\int_{\mathcal{H}}\int_{\mathcal{H}}e^{-r\|x-y\|^{2}}d\mu(x)d\mu(y)d\beta_{c}(r), \quad c>0.
\end{equation}
The continuous and bounded function $I_{\mu}(r):=\int_{\mathcal{H}}\int_{\mathcal{H}}e^{-r\|x-y\|^{2}}d\mu(x)d\mu(y)$, $r\geq0$, is  positive for every $r>0$ by \cite{gaussinfi}, additionally Equation \ref{nonposgauss} implies that ($c=s+1$)
$$
0= \int_{[0, \infty)}e^{-sr} I_{\mu}(r)d\beta_{1}(r), \quad  s \geq 0. 
$$
By the uniqueness representation of Laplace transform, this can only occur if the finite measure $ I_{\mu}d\beta_{1}$ is the zero measure on $[0, \infty)$. The behaviour of $I_{\mu}$ implies that this occur if and only if $I_{\mu}(0)=0$ and $\beta_{1}$ is a multiple of $\delta_{0}$, the latter implies that $\psi$ is a constant function, which is a contradiction. 
\end{proof}
\begin{proof}[\textbf{Proof of Lemma \ref{kobolnew2}}] By Theorem \ref{kobolnewneural} we only need to focus on the finite dimensional case.  We prove $(i)$ and $(ii)$ by  showing that is sufficient to prove the case $\ell=1,2$, which will follow from $(iii)$ and $(iv)$. For the induction  argument on $(ii)$ we assume a more general setting, that $\psi(t)= t^{\ell-1}\log (t) + bt^{\ell-1}$, with $b \in \mathbb{R}$.\\
Indeed, suppose that $\ell \geq 3$. Note then that the function $y \in \mathcal{H} \to F(y):=\psi(\|x-y\|^{2}) \in \mathbb{R}$ is twice differentiable on each direction of an orthonormal basis  $(e_{\iota})_{\iota \in \mathfrak{I}}$ for $\mathcal{H}$, and
$$
\frac{\partial^{2}F}{\partial^{2} e_{\iota}}(y) = 4\psi^{\prime \prime}(\|x-y\|^{2})(y_{\iota} - x_{\iota})^{2} + 2\psi^{\prime}(\|x-y\|^{2}).
$$
Since $\psi \in C^{\ell-1}([0, \infty)) \cap CM_{\ell}$ (or $-\psi$ is an element, the sign does not make difference for the induction step), we have that $|\psi^{\prime}(t)| \lesssim (1+t)^{\ell-1}$ and $|\psi^{\prime \prime}(t)| \lesssim (1+t)^{\ell-2}$. In particular, the second derivative  is a function in  $L^{1}(|\mu|)$ and summing on the  $\iota$ variable we obtain ($m= dim(\mathcal{H})$)
\begin{equation}\label{recursion}
0= \int_{\mathcal{H}}4\psi^{\prime \prime}(\|x-y\|^{2})\|x-y\|^{2} + 2m\psi^{\prime}(\|x-y\|^{2})d\mu(x), \quad y \in \mathcal{H}. 
\end{equation}
When  $\psi$ is a function of type $(i)$ or $(ii)$, the integrand on this equation is equal to a positive multiple of $\|x-y\|^{2a-2}$ (or  $\|x-y\|^{2\ell-4}\log (\|x-y\|^{2})$ plus  a multiple of $\|x-y\|^{2\ell-4}$), which is the induction argument.\\
Now, let $\psi$ be an arbitrary function on $CM_{\ell}$, $\ell=1,2$, that is not a polynomial. For every $t >0$, define $\eta_{t} := t\mu - \tau_{t}$, where $ \tau_{t} =t\mu(\mathcal{H})\delta_{0} - (\delta_{t v_{\mu}} - \delta_{-tv_{\mu}})/2$ and  $v_{\mu}$ is the  vector mean, that is
$$
\int_{\mathcal{H}}\langle x,y \rangle d\mu(x) = \langle v_{\mu}, y  \rangle, \quad y \in \mathcal{H}.
$$
On the case $\ell=1$ the vector $v_{\mu}$ might not be well defined, on this case define it as the vector zero. Then $\eta_{t}(\mathcal{H})=0$, and if it is well defined $v_{\eta_{t}}=0$. By the hypothesis we  obtain that
\begin{align*}
 4\int_{\mathcal{H}}\int_{\mathcal{H}} \psi(\|x-y\|^{2})d\eta_{t}(x)d\eta_{t}(y) &= \int_{\mathcal{H}}\int_{\mathcal{H}} \psi(\|x-y\|^{2})d2\tau_{t}(x)d2\tau_{t}(y)\\
&=\psi(0)( 4t^{2}\mu(\mathcal{H})^{2} +2) - 2 \psi(4t^{2}\|v_{\mu}\|^{2}) 
\end{align*}
By Theorem \ref{principal}, this is a nonnegative number for every $t>0$.\\
On the other hand, if $\ell=2$ by the relation on Equation \ref{compleelltimes2}, we know that $(-1)^{2}\psi(t)$ converges to $+\infty$ as $t\to \infty$, so if $\|v_{\mu}\|\neq 0$ or $\psi(0)<0$ we would reach a contradiction, consequently $v_{\mu}=0, \psi(0)= 0$. In particular,  we obtain that the double integral with respect to $\eta_{1}$ is zero, which by Theorem \ref{principal}   we must have that $\mu= \mu(\mathcal{H})\delta_{0}$, because $\psi$ is not a polynomial. From this equality and the initial assumption on $\mu$ we obtain that $\mu(\mathcal{H})\psi(\|y\|^{2})=0$ for every $y \in \mathcal{H}$, which can only occur if $\mu$ is the zero measure because $\psi$ is not a polynomial.\\ 
The case $\ell=1$ follows by a similar analysis.
\end{proof}

\bibliographystyle{siam}
\bibliography{Referrences.bib}

\end{document}